%% file: alta-v03.tex
\documentclass[reqno,12pt]{amsart}
\usepackage{amsmath,amssymb,epsfig,color,slashbox,url}

\numberwithin{equation}{section}

\newtheorem{theorem}{Theorem}[section]

\newtheorem{lemma}[theorem]{Lemma}
\newtheorem{proposition}[theorem]{Proposition}
\newtheorem{definition}[theorem]{Definition}
\newtheorem{corollary}[theorem]{Corollary}

\newtheorem{remark}[theorem]{Remark}

\newcommand{\vanish}[1]{}

\newcommand{\asc}{\xrightarrow{a}}
\newcommand{\desc}{\xrightarrow{d}}

\newcommand{\lawalk}{\leq_{la}}
\newcommand{\rawalk}{\leq_{ra}}

\newcommand{\srawalk}{<_{ra}}

\begin{document}

\title[Alt-acyclic tournaments]{Alternation acyclic tournaments}

\author[G\'abor Hetyei]{G\'abor Hetyei}

\address{Department of Mathematics and Statistics,
  UNC-Charlotte, Charlotte NC 28223-0001.
WWW: \tt http://www.math.uncc.edu/\~{}ghetyei/.}



\date{\today}

\begin{abstract}
We define a tournament to be alternation acyclic if it does not contain
a cycle in which descents and ascents alternate. Using a result by
Athanasiadis on hyperplane arrangements, we show that these tournaments
are counted by the median Genocchi numbers. By establishing a bijection
with objects defined by Dumont, we show that alternation acyclic
tournaments in which at least one ascent begins at each vertex, except
for the largest one, are counted by the Genocchi numbers of the first
kind. Unexpected consequences of our results include a pair of ordinary
generating function formulas for the Genocchi numbers of both kinds and
a simple model for the normalized median Genocchi numbers.   
\end{abstract}
\maketitle

\section*{Introduction}

Genocchi numbers of the first kind are closely related to the Bernoulli
and Euler (tangent and secant) numbers, and the first classes of
permutations counted by them, introduced by Dumont~\cite{Dumont} are
{\em alternating} in one way or another, just like the alternating
permutations, counted by the tangent and secant numbers. Whereas the
tangent and secant numbers found a geometric interpretation through the work of
Purtill~\cite{Purtill}, Stanley~\cite{Stanley-flag} and many people in
their wake (using Andr\'e permutations, first studied by Foata,
Sch\"utzenberger  and Strehl in the 1970-ties~\cite{Foata}), there seems
to be far less done in terms of finding geometric interpretations for
the various types of Genocchi numbers, studied concurrently
with the Genocchi numbers of the first kind. A notable exception is the work
of Feigin~\cite{Feigin-deg}, identifying the Poincar\'e polynomials of
the complete degenerate flag-varieties as $q$-generalizations of the
normalized median Genocchi numbers introduced in \cite{Han-Zeng}.

This paper proposes a geometric interpretation of the Genocchi numbers,
in the world of hyperplane arrangements. We arrive at this
interpretation by generalizing the definition of semiacyclic tournaments,
used by Postnikov and Stanley~\cite{Postnikov-Stanley}, and
independently, Shmulik Ravid, to bijectively label the regions created
by the Linial arrangement. The subject of this paper is this wider class of
tournaments (we call them {\em alternation acyclic}), which may be used to
bijectively label the regions in a homogeneous variant of the Linial
arrangement, which we call the {\em homogenized Linial
  arrangement}. The Linial arrangement studied in the literature is a 
section of our homogenized Linial arrangement. Using the technique of
counting points in vector spaces over finite fields, developed by
Athanasiadis~\cite{Athanasiadis-charpol}, we are able to prove that the
number of regions created by our homogenized Linial arrangement, and
thus the number of alternation acyclic tournaments, is a median Genocchi
number (Theorem~\ref{thm:all-alta}). No direct combinatorial argument
was found for this result. On 
the other hand, using this result it is 
possible to find a simple class of objects counted by the median
Genocchi numbers, which allow a simple ${\mathbb Z}_2^n$-action,
making the known fact transparent, that the median Genocchi number $H_{2n+1}$
is an integer multiple of $2^n$. The set of ${\mathbb Z}_2^n$-orbits
also has a simple combinatorial representation
(Theorem~\ref{thm:nmGenocchi}). While this work was 
still a preprint, A.\ Bigeni has found a highly nontrivial
bijection~\cite{Bigeni-bij} between this model and the model of
Feigin~\cite{Feigin-deg}.

We also obtain an explicit combinatorial argument showing that {\em
ascending} alternation acyclic tournaments (in which each numbered
element defeats at least one element with a larger number, except for
the largest numbered element), are counted by the Genocchi numbers of
the first kind (Corollary~\ref{cor:asc}). The extension of this direct
counting approach to all alternation acyclic tournaments yields
recurrences leading to formulas for the ordinary generating functions
for the Genocchi numbers of the first and second kinds.     
 
Our paper is structured as follows. After collecting basic facts about
Genocchi numbers, hyperplane arrangements in general and the Linial
arrangement in particular, we introduce alternation acyclic tournaments
in Section~\ref{sec:alta} and prove their most important properties. In
particular, we show that they induce a partial order which we call the
right alternating walk order. In Section~\ref{sec:forest} we show how to
encode each alternation acyclic tournament with a pair $(\pi,p)$, where
the permutation $\pi$ is a linear extension of the
alternating walk order, and the parent function $p$ assigns to each
element a larger element or the infinity symbol as its parent, thus
defining a partial order that is a forest. Even though this
representation is not unique, using it allows us to introduce a
homogenized generalization of the Linial arrangement
in Section~\ref{sec:alth} and show that the regions of this hyperplane
arrangement are in bijection with our alternation acyclic
tournaments. Section~\ref{sec:alth} also contains the proof of
Theorem~\ref{thm:all-alta}, stating 
that the number of all alternation acyclic tournaments is a median
Genocchi number. We take a closer look at the codes $(\pi,p)$ in
Section~\ref{sec:lmax} and find a way to select unique codes (which we
call largest maximal representations) for each alternation acyclic
tournament. We also obtain a characterization of all valid codes.
In Section~\ref{sec:refined} we use this characterization to obtain
Theorem~\ref{thm:Dumontfine}, a combinatorial result
refining our counting of alternation acyclic tournaments. The key ingredient to obtain
this result is a descent-sensitive coding of permutations, using
excedant functions, a variant of an idea already present in
Dumont's work~\cite{Dumont}. This result allows counting ascending
alternation acyclic tournaments with the help of Dumont's theorem, and
introduce new combinatorial models for the median and normalized median Genocchi
numbers in Section~\ref{sec:genmod}. The generating function formulas
are derived in Section~\ref{sec:genfun}. This paper raises as many
questions as it answers: some of these are mentioned in the concluding
Section~\ref{sec:concl}.

\section{Preliminaries}

\subsection{Genocchi numbers}

The {\em Genocchi numbers $G_n$ of the first kind} are given by the
exponential generating function
$$
\sum_{n=1}^{\infty} G_n \frac{t^n}{n!}=\frac{2t}{e^t+1}.
$$
The only nonzero $G_n$ with $n$ odd is $G_{1}=1$. 
For even $n$, the first few values of $G_n$ are $G_{2}=-1$, $G_{4}=1$,
$G_{6}=-3$, $G_{8}=17$, $G_{10}=-155$ and $G_{12}=2073$.
Their study goes back at least to Seidel~\cite{Seidel}, who published a
triangular table, called {\em Seidel's triangle},  allowing to compute
them recursively. Generalizations and variants of Seidel's triangle
include~\cite{Ehrenborg,Zeng}.
The first combinatorial models for them were given by
Dumont~\cite{Dumont}. We will use the following result from his
work~\cite[Corollaire du Th\'eor\`eme 3]{Dumont}, which characterizes
the signless Genocchi numbers as numbers of {\em excedant} functions. A
function, defined on a set of integers, is {\em excedant} if it satisfies
$f(i)\geq i$ for all $i$. Note that excedant functions are also called
{\em surjective pistols} by Dumont and Viennot in~\cite{Dumont-Viennot}.
\begin{theorem}[Dumont]
  \label{thm:dumont}
The unsigned Genocchi number $|G_{2n+2}|$ is the number of excedant 
functions $f:\{1,\ldots,2n\}\rightarrow \{1,\ldots,2n\}$ satisfying
$f(\{1,\ldots,2n\})=\{2,4,\ldots,2n\}$.   
\end{theorem}  
The following wording is easily seen to be equivalent.
\begin{corollary}
  \label{cor:dumont}
The unsigned Genocchi number $|G_{2n+2}|$ is the number of ordered pairs 
$$((a_1,\ldots,a_{n}), (b_1,\ldots,b_{n}))\in {\mathbb Z}^{n}\times
{\mathbb Z}^{n}$$
such that $1\leq a_i, b_i\leq i$ hold for all $i$
and the set $\{a_1,b_1,\ldots,a_{n},b_{n}\}$ equals $\{1,\ldots,n\}$.
\end{corollary}
Another model, ``alternating pistols'', may be found
in~\cite{Viennot}. For more information on
the Genocchi numbers of the first kind we refer the reader to its
entries (A036968 and A001469) in~\cite{OEIS}. 

The {\em median Genocchi numbers $H_{2n-1}$}, also called {\em Genocchi
  numbers of the second kind}, also appear already in Seidel's
triangle. Their study evolved concurrently with the study of the
Genocchi numbers of the first kind. For detailed bibliography on them
we refer the reader to the above mentioned sources, and their entry
(A005439) in~\cite{OEIS}. Their first few values are $H_1=1$, $H_3=2$,
$H_5=8$, $H_7=56$, $H_9=608$ and $H_{11}=9440$. 
In this paper we will use the following recent
result on them, due to Claesson, Kitaev, Ragnarsson and
Tenner~\cite{Claesson-LS}: 
\begin{equation}
\label{eq:HLS}  
H_{2n-1}=\sum_{k=1}^n (-1)^{n-k}\cdot (k!)^2\cdot PS_n^{(k)}. 
\end{equation}
Here the numbers $PS_n^{(k)}$ are the {\em Legendre-Stirling numbers},
see the work of Andrews, Gawronski and Littlejohn~\cite{Andrews-LS}. 

The median Genocchi number $H_{2n-1}$ is known to be an integer multiple
of $2^{n-1}$, see \cite{Barsky-Dumont}. The numbers $h_n=H_{2n+1}/2^n$
are the {\em normalized median Genocchi numbers}. They are listed as
sequence A000366 in~\cite{OEIS}. The first few values are $h_1=1$, $h_2=2$,
$h_3=7$, $h_4=38$, $h_5=295$, $h_6=3098$ and  $h_7=42271$. Several combinatorial
  models of these numbers exists, perhaps the most known are the {\em
    Dellac configurations}~\cite{Dellac}. Other models may be found in
  the works of Bigeni~\cite{Bigeni},
  Feigin~\cite{Feigin-deg,Feigin-med}, Han and Zeng~\cite{Han-Zeng}, and
  Kreweras and Barraud~\cite{Kreweras-Barraud}. 
We will present a new combinatorial model for the normalized median Genocchi
numbers in Theorem~\ref{thm:nmGenocchi}. Recently
Bigeni~\cite{Bigeni-bij} found a highly nontrivial bijection between our
model and Feigin's model~\cite{Feigin-deg}. 

\subsection{Hyperplane arrangements}

A hyperplane arrangement ${\mathcal A}$ is a finite collection of
codimension one hyperplanes in a $d$-dimensional vector space over 
${\mathbb R}$, which partition the space into regions. The number
$r({\mathcal A})$ of these regions may be found using
{\em Zaslavsky's formula}~\cite{Zaslavsky}, stating 
\begin{equation}
  \label{eq:Zaslavsky}
r({\mathcal A})=(-1)^d \chi({\mathcal A},-1). 
\end{equation}
Here $\chi({\mathcal A},q)$ is the {\em characteristic polynomial} of
the arrangement, which Zaslavsky expressed in terms of the M\"obius
function in the intersection lattice of the hyperplanes. Instead of 
using Zaslavsky's original formulation, we will use the following
result of Athanasiadis~\cite[Theorem 2.2]{Athanasiadis-charpol}.
In the case when the hyperplanes of ${\mathcal A}$ are defined by
equations with integer coefficients, we may consider the hyperplanes
defined by the same equations in a vector space of the same dimension
over a finite field ${\mathbb F}_q$ with $q$ elements, where $q$ is a
prime number. If $q$ is sufficiently large, then the number 
$\chi({\mathcal A},q)$ is the number of points in the vector space that
do not belong to any hyperplane in the arrangement:  
\begin{equation}
\label{eq:Christos}
\chi({\mathcal A},q)=\left|{\mathbb F}_q^d -\bigcup {\mathcal A}\right|.
\end{equation}

\subsection{Semiacyclic tournaments and the Linial arrangement}

This paper is on a class of directed graphs properly containing the
class of {\em semiacyclic tournaments}. A tournament $T$ on the set
$\{1,\ldots,n\}$ is a directed graph with no loops nor multiple edges,
such that for each pair of 
vertices $\{i,j\}$ from  $\{1,\ldots,n\}$, exactly one of the directed
edges $i\rightarrow j$ and $j\rightarrow i$ belongs to the graph. We
consider $\{1,\ldots,n\}$ with the natural order on positive integers. 
A directed edge $i\rightarrow j$ in a cycle is an {\em ascent} if $i<j$
otherwise it is a {\em descent}. An {\em ascending cycle} is a directed
cycle in which the number of descents does not exceed the number of
ascents. A tournament on $\{1,\ldots,n\}$ is {\em semiacyclic} if it
contains no ascending cycle. 

Semiacyclic tournaments arose in the study of the {\em Linial arrangement ${\mathcal
    L}_{n-1}$}. This arrangement is the set of hyperplanes  
$$
x_i-x_j=1 \quad \mbox{for $1\leq i<j\leq n$}
$$
in the subspace $V_{n-1}\subset {\mathbb R}^{n-1}$ given by
$$V_{n-1}=\{(x_1,\ldots,x_n)\in {\mathbb R}^{n}\:
: \: x_1+\cdots +x_n=0\}.$$ 
To each region $R$ in ${\mathcal L}_{n-1}$ we may associate a tournament 
on $\{1,\ldots,n\}$ as follows: for each $i<j$ we set $i\rightarrow j$
if $x_i>x_j+1$ and we set $j\rightarrow i$ if $x_i<x_j+1$. 
Postnikov and Stanley~\cite[Proposition 8.5]{Postnikov-Stanley},
and independently Shmulik Ravid, observed that, the correspondence
above establishes a bijection between the regions of
the Linial arrangement ${\mathcal L}_{n-1}$ and the set of
semiacyclic tournaments on the set $\{1,\ldots,n\}$.

\section{Alternation acyclic tournaments}
\label{sec:alta}

In this section we define alternation acyclic tournaments and show some
of their most basic properties. We define ascents and descents
essentially the same way as Postnikov and Stanley do it
in~\cite{Postnikov-Stanley}. The only minor difference is that we will 
use the notion of an ascent and a descent on all edges, not only on
those contained in a directed cycle.  
\begin{definition}
We call the arrow $i\rightarrow j$ an {\em ascent} if $i<j$ holds,
otherwise we call it a {\em descent}. We will use the notation $i\asc j$,
respectively $i\desc j$, to denote ascents and descents respectively.
A directed cycle $C=(c_0,c_1,\ldots,c_{2k-1})$ is {\em alternating} if ascents 
and descents alternate along the cycle, that is, $c_{2j}\desc
c_{2j+1}$ and $c_{2j+1}\asc c_{2j+2}$ hold for all $j$ (here we identify
all indices modulo $2k$). A tournament is {\em alternation acyclic (or
  alt-acyclic)} if it contains no alternating cycle.  
\end{definition}

Clearly an alternating cycle is also an ascending cycle, hence
every semiacyclic tournament is also alternation acyclic. In
Section~\ref{sec:alth} we will state an 
analogue of~\cite[Proposition 8.5]{Postnikov-Stanley} for alt-acyclic
tournaments, and we will explain how each Linial arrangement is a
section of a hyperplane arrangement whose regions are labeled with
alt-acyclic tournaments.  

A generalization of the notion of a directed cycle is a the notion of a
directed closed walk, in which revisiting vertices is allowed. The
following, important observation implies that, for tournaments,
excluding alternating closed walks vs. alternating cycles makes no
difference.
\begin{proposition}
  \label{prop:4cycle}
Suppose a tournament $T$ on $\{1,\ldots,n\}$ contains a closed
alternating walk $(c_0,c_1,\ldots,c_{2k-1})$, that is, a closed walk, in
which descents and ascents alternate. Then $T$ contains an alternating
cycle of length $4$. 
\end{proposition}  
\begin{proof}
Let $(c_0,c_1,\ldots,c_{2m-1})$ be an alternating closed walk
of minimum length in $T$. It suffices to show that we must have
$m=2$. Indeed, note that a closed alternating walk must have even
length, and there is no closed walk $c_0\rightarrow c_1\rightarrow c_0$
in a tournament, so we must have $m\geq 2$. Note also that a closed walk of
length $4$ must visit $4$ distinct vertices, as it can not be the
composition of a closed walk of length $3$ and one additional edge.

Assume, by contradiction, that we have $m\geq 3$.  As
usual, we will identify the indices modulo $2m$. Furthermore, without
loss of generality, we will assume that
the arrows $c_{2i}\rightarrow c_{2_{i+1}}$ are descents and the
arrows $c_{2i-1}\rightarrow c_{2i}$ are ascents.

It suffices to show that in such a closed alternating walk we must have
$c_{2i}<c_{2i+4}$ for all $i$. Since we assumed $m\geq 3$, this will
yield a contradiction of the form $c_0<c_4<\cdots <c_0$. We will
distinguish two cases:

{\bf\noindent Case 1:} $c_{2i}=c_{2i+3}$ holds. In this case the
statement follows from the fact that $c_{2i+3}\rightarrow c_{2i+4}$ is
an ascent.  

{\bf\noindent Case 2:} $c_{2i}\neq c_{2i+3}$ holds. In this case it
suffices to show that we must have $c_{2i}< c_{2i+3}$: the statement
follows then by transitivity from $c_{2i+3}\asc c_{2i+4}$. 
Assume, by contradiction, that $c_{2i}>c_{2i+3}$
holds. Then either we have $c_{2i+3}\asc c_{2i}$ and $(c_{2i}, c_{2i+1}, 
c_{2i+2}, c_{2i+3})$ is an alternating $4$-cycle, or we have
$c_{2i}\desc c_{2i+3}$ and we may use this descent to replace the subwalk
$(c_{2i}\desc c_{2i+1} \asc c_{2i+2}\desc c_{2i+3})$ in the closed walk, thus
obtaining a shorter alternating closed walk. In either case, we reach a
contradiction with the minimality of $m$.
\end{proof}  

The following consequence of Proposition~\ref{prop:4cycle} is analogous
to a result by Postnikov and Stanley~\cite[Theorem 8.6]{Postnikov-Stanley} 
which characterizes semiacyclic tournaments as tournaments containing no
ascending cycle of length at most $4$.
\begin{corollary}
  \label{cor:4cycle}
A tournament $T$ on $\{1,\ldots,n\}$ is alternation acyclic if and only
if it contains no alternating cycle of length $4$.   
\end{corollary}

Another way to characterize alternation acyclic tournaments is to
describe them in terms of the right-alternating walk relation.
\begin{definition}
In a tournament $T$ on $\{1,\ldots,n\}$,
there is a {\em right-alternating walk} from $u$ to $v$ if $u=v$ or
there is a directed walk $u=w_0\desc w_1 \asc w_2 \desc \cdots \desc
w_{2i-1}\asc w_{2i}=v$ from $u$ to $v$ in which descents and ascents
alternate, the first edge being a descent and the last edge being an
ascent. We will use the notation $u\rawalk v$ when there is a
right-alternating walk from $u$ to $v$, and we will refer to $\rawalk$
as the {\em right-alternating walk order induced by $T$}. We will also
use the shorthand notation $u\srawalk w$ when $u\rawalk v$ and $u\neq v$
hold.   
\end{definition}  

\begin{proposition}
  \label{prop:rawalk}
A tournament $T$ on $\{1,\ldots,n\}$ is alternation acyclic, if and
only the induced right-alternating walk order is a partial order.   
\end{proposition}  
\begin{proof}
The relation $\rawalk$ is by definition reflexive and it is obviously
transitive, as the concatenation of right-alternating walks yields a
right-alternating walk. Hence the relation $\rawalk$ is a partial order
if and only if it is antisymmetric. This property is easily seen to be
equivalent to the non-existence of a nontrivial closed alternating walk,
whose non-existence is equivalent to the non-existence of an alternating
$4$-cycle by Proposition~\ref{prop:4cycle}. As noted in
Corollary~\ref{cor:4cycle}, the non-existence of an
alternating $4$-cycle is equivalent to the tournament being alt-acyclic.       
\end{proof}

\begin{remark}
{\em There is an apparent asymmetry in the definition of the
right-alternating walk order. One could analogously define the
left-alternating walk order $\lawalk$ using alternating walks that
begin with an ascent and end with a descent. It is similarly easy to see
the analogue of Proposition~\ref{prop:rawalk} stating that a tournament
is alt-acyclic, if and only if $\lawalk$ is a partial order. It should
be noted that the class of alternation acyclic tournaments is closed
under reversing all directed edges and it is also closed under
renumbering the vertices such that each $i\in \{1,\ldots,n\}$ is
replaced by $n+1-i$. Under each of these operations, the role of the
partial order $\rawalk$ is taken over by the partial order $\lawalk$ and
vice versa.}  
\end{remark}

\section{Representing alt-acyclic tournaments as biordered forests}
\label{sec:forest}

A partially ordered set is a {\em forest} if every element is covered by
at most one element. A formula counting linear extensions of a forest is
due to Knuth~\cite{Knuth}. For a bibliography on generalizations
and recent results we refer the reader to~\cite{Hivert-Reiner}. Note
that Hivert and Reiner use the dual definition of a forest, in which
every element covers at most one element. We follow the definition of
Bj\"orner and Wachs~\cite{Bjorner-Wachs}. In this section we will show
that every alternation acyclic tournament may be represented as a
tournament induced by a biordered forest, where one of the orders is a
linear extension, and the other one is an arbitrary permutation. 
We will think of the linear extension as a numbering of the elements
from $1$ to $n$, and we will encode the second numbering by a word
$\pi=\pi(1)\pi(2)\cdots \pi(n)$, where the label of $j\in\{1,\ldots,n\}$ is
the position $\pi^{-1}(j)$ of the number $j$ in $\pi$

If an element $i$ is covered by an element $j$ in a forest, we will
write $j=p(i)$ and say that $j$ is the {\em parent} of $i$. We will also use the
notation $p(i)=\infty$ when $i$ has no parent, and we will say that $i$
is a {\em root}. In fact, the Hasse diagram will be a union of trees,
and the roots will be exactly the maximum elements of these
trees. Marking the root of each tree defines the partial order.  
We fix a linear extension of a forest, by numbering its
elements in increasing order from $1$ to $n$, where $n$ is the number
of the elements. The parent function $p: \{1,2,\ldots,n\}\rightarrow
\{2,\ldots,n\}\cup\{\infty\}$ must then satisfy $i<p(i)$ 
for all $i\in \{1,2\ldots,n\}$. The converse is also true:
\begin{proposition}
\label{prop:ext1}
Given a forest on the set $\{1,2,\ldots,n\}$, defined by the parent
function $p:\{1,2,\ldots,n\}\rightarrow \{2,\ldots,n\}\cup\{\infty\}$, the
order $1<2<\cdots <n$ is a linear extension of the forest, if and
only if the parent function satisfies $i<p(i)$ for all $i\in
\{1,2,\ldots,n\}$.  
\end{proposition}
\begin{proof}
If the numbering represents a linear extension, then the
condition $i<p(i)$ is necessary. Conversely, assume the function $p$
satisfies the stated property. If $i$ is less than $j$ in
the order of the forest, then for the length $\ell$ of the shortest
path from $i$ to $j$ in the Hasse diagram we have $j=p^{\ell}(i)$,
implying $i<p(i)<p(p(i))<\cdots <p^{\ell}(i)=j$.
\end{proof}

From now on we will identify each element of the forest with its label in
a fixed linear extension, and we encode the forest with its parent
function $p:\{1,2,\ldots,n\}\rightarrow
\{2,\ldots,n\}\cup\{\infty\}$. Next we give a second labeling of the
vertices in terms of the {\em inverse} of a permutation $\pi$ of the set
$\{1,\ldots,n\}$.  
\begin{definition}
\label{def:pospi}  
Given a permutation  $\pi: \{1,2,\ldots,n\}\rightarrow
\{1,2,\ldots,n\}$, we will say that the {\em labeling induced by the
  positions in $\pi$}
is the labeling that associates to each $i\in \{1,2,\ldots, n\}$ the
position $\pi^{-1}(i)$ of $i$ in $\pi$.
\end{definition}  
\begin{definition}
\label{def:biordered} We will refer to an ordered triplet of a forest,
one of its linear extensions and an arbitrary numbering of its
elements as a {\em biordered forest}. We will encode the forest and its
linear extension with the corresponding parent function
$p:\{1,2,\ldots,n\}\rightarrow \{2,\ldots,n\}\cup\{\infty\}$, satisfying
$p(i)>i$ for all $i$, and the second numbering by the permutation $\pi$
whose positions induce the second numbering. We will refer to the pair
$(\pi,p)$ as the {\em code} of the biordered forest. 
\end{definition}  
The following statement is obvious.
\begin{proposition}
  \label{prop:code}
The correspondence described in Definition~\ref{def:biordered}
establishes a bijection between ordered triplets formed by a forest on
$n$ elements, a linear extension of this forest, and an arbitrary
numbering of its elements and ordered pairs $(\pi,p)$ formed by a
function  $p:\{1,2,\ldots,n\}\rightarrow
\{2,\ldots,n\}\cup\{\infty\}$, satisfying $p(i)>i$ for all $i$ and by a
permutation $\pi$ of $\{1,2,\ldots,n\}$. 
\end{proposition}  
\begin{remark}
{\em We could also use the language of $(P,\omega)$-partitions, see Stanley's
book~\cite{Stanley-EC1/2}. This begins with considering a partially
ordered set (in our case: a forest) and a bijection $\omega: P\rightarrow
\{1,2,\ldots,n\}$. Stanley calls the labeling $\omega$ {\em natural}
when $\omega$ is a linear extension of $P$. In such terms, we consider
forests with a pair of labelings, one of them natural, the other one
arbitrary.} 
  \end{remark}
Next we define a tournament induced by a biordered forest.
\begin{definition}
Let $(\pi,p)$ be the code of a biordered forest on $n$ elements. We
define the tournament $T=T(\pi,p)$ as {\em the tournament induced by
  the biordered forest} to be the tournament whose vertex set is
$\{1,2,\ldots, n\}$ and whose directed edges are defined as follows. For
all $u<v$ we set $u\asc v$ if $p(u)\neq \infty$ and $\pi^{-1}(v)\geq
\pi^{-1}(p(u))$ hold, otherwise we set $v\desc u$.
\end{definition}  
We may visualize the pair $(\pi,p)$ as an arc-diagram, shown in
Figure~\ref{fig:arcs}.

\begin{figure}[h]
\begin{center}
\input{fig01_pspdf.tex}
\end{center}
\caption{Arc representation of a pair $(\pi,p)$}
\label{fig:arcs}
\end{figure}
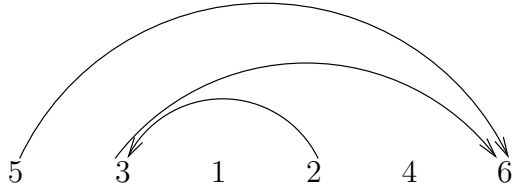

We line up the vertices of the tournament left to right, in the order
$\pi(1)$, $\pi(2)$, \ldots, $\pi(n)$. The permutation $\pi$ in
Figure~\ref{fig:arcs} is $531246$. Next, for each $i$ such that
$p(i)\neq \infty$, we draw a directed arc from $i$ to 
$p(i)$. For example, in the picture there is an
arc from $\pi(4)=2$ in position $4$ to $\pi(2)=3$ in position
$2$, indicating $p(2)=3$. The number $3$ is the leftmost number larger
than $2$ for which $2\asc 3$. All numbers larger than $2$ that are to
the left of $3$ defeat $2$, and $2$ defeats all numbers larger than $2$
to the right of $3$. Hence we have $5\desc 2$, $2\asc 3$, $2\asc 4$
and $2\asc 6$.  Similarly we have $p(3)=6$ and so the only ascent
starting at $3$ is $3\asc 6$.  The parent of the numbers $\pi(3)=1$,
$\pi(5)=4$ and $\pi(6)=6$ is $\infty$, no arc begins at these vertices,
no ascent starts at these vertices. The parent function $p$ defines a
forest with three connected components, the roots of these three trees
are $1$, $4$ and $6$, respectively. 
  
The next two statements explain how biordered forests are related to
alt-acyclic tournaments. 

\begin{proposition}
  \label{prop:arcs-alta}
Every biordered forest $(\pi,p)$ induces an alternation acyclic tournament
$T$. Furthermore, the permutation $\pi$ is a linear extension of the
right alternating walk order induced by $T$.  
\end{proposition}  
\begin{proof}
Let $(\pi,p)$ be the code of the biordered forest on $n$ elements, and
let $T=T(\pi,p)$ be the tournament induced by it. First we show that $T$
is alt-acyclic. By  Corollary~\ref{cor:4cycle} it suffices to show that
there is no alternating cycle of length $4$ in $T$. Assume, by
contradiction, that $u_1\asc u_2\desc u_3\asc u_4\desc u_1$ holds for
some $\{u_1,u_2,u_3,u_4\}\subseteq \{1,2,\ldots,n\}$. Since we have
$u_1\asc u_2$ and $u_4\desc u_1$ we obtain that $u_4$ must appear to the
left of $p(u_1)$ whereas $u_2$ can not appear to the left of
$p(u_1)$ in $\pi$. In other words, we must have
$$\pi^{-1}(u_4)<\pi^{-1}(p(u_1))\leq \pi^{-1}(u_2).$$
Similarly $u_2\desc u_3$ and $u_3\asc u_4$ imply that we must have
$$\pi^{-1}(u_2)<\pi^{-1}(p(u_3))\leq \pi^{-1}(u_4).$$
This is a contradiction, as $u_2$ and $u_4$ can not mutually precede
each other in $\pi$.

To show the second part of the statement it suffices to prove that
$\pi^{-1}(u)<\pi^{-1}(v)$ holds whenever $v$ covers $u$ in the right
alternating walk order. For an arbitrary pair $(u,v)$, satisfying
$u\rawalk v$, the statement $\pi^{-1}(u)\leq \pi^{-1}(v)$ follows then by
considering a saturated chain from $u$ to $v$. If $v$ covers $u$
in the right alternating walk order, then there is a $w$ such that
$u\desc w\asc v$ holds. By the definition of $T(\pi,p)$, $u$ must be to
the left of $p(w)$, whereas $v$ can not be to the left of $p(w)$ in
$\pi$. In other words, $\pi^{-1}(u)<\pi^{-1}(p(w))\leq \pi^{-1}(v)$ must
hold, and $u$ is to the left of $v$ in $\pi$.
\end{proof}  

Our next result is the converse of Proposition~\ref{prop:arcs-alta}.

\begin{proposition}
\label{prop:alta-arc}
Let $T$ be an alternation acyclic tournament on $\{1,2,\ldots, n\}$, and
let $\pi$ be any linear extension of the right alternating walk
order. Then there is a unique parent function
$p:\{1,2,\ldots,n\}\rightarrow \{2,\ldots,n\}\cup\{\infty\}$ such that
the tournament induced by $(\pi, p)$ is $T$. 
\end{proposition}
\begin{proof}
Clearly, for each $u\in \{1,\ldots,n\}$, the only way to define $p(u)$
is to set $p(u)$ equal to the leftmost $v$ in $\pi$ such that $u\asc v$
holds, if such a $v$ exists, and to set $p(u)=\infty$ when no ascent
begins at $u$. We only need to verify that the tournament $T(\pi,p)$
induced by the biordered forest with code $(\pi,p)$ is the
same as the tournament $T$ we started with. Consider a pair $(u,v)$ of
vertices satisfying $u<v$. If $p(u)=\infty$ or
$v$ is to the left of $p(u)\neq \infty$ then, by the definition of $p$,
we must have $v\desc u$ in $T$ and also in $T(\pi,p)$. Also by
definition, $u\asc v$ holds in both tournaments, if $v=p(u)$. We are left to
consider the case when $v$ is to the right of $p(u)$ in $\pi$. In
$T(\pi, p)$ we must have $u\asc v$, the only remaining question is,
could we have $v\desc u$ in the tournament $T$, for such a vertex $v$? The
answer is no, since $v\desc u\asc p(u)$ would imply $v\srawalk p(u)$ in
contradiction with $\pi$ being a linear extension of the partial order
$\rawalk$.   
\end{proof}   

\begin{remark}
{\em For any alt-acyclic tournament $T$, the element $1$ is always
incomparable to the other elements of $\{1,\ldots,n\}$ in the right
alternating walk order, hence the partial order $\rawalk$ has always at
least two linear extensions. This makes the use of biordered forests to
directly count alt-acyclic tournaments difficult. We will see two
different ways to overcome this difficulty in Sections~\ref{sec:alth}
and \ref{sec:lmax}.}    
\end{remark}

\section{Counting alternation acyclic tournaments using hyperplane
  arrangements}
\label{sec:alth}

In this section we introduce a hyperplane arrangement whose regions are
in bijection with the alternation acyclic tournaments. Using a result of
Athanasiadis~\cite{Athanasiadis-charpol}, we will be able to count
them.

\begin{definition}
Consider the vector space ${\mathbb R}^{2n-1}$ with coordinate functions
$x_1$,\ldots,$x_n$, $y_1$,\ldots,$y_{n-1}$.   
We define the {\em homogenized Linial arrangement ${\mathcal H}_{2n-2}$}
as the set of hyperplanes
\begin{equation}
\label{eq:hL}  
x_i-x_j=y_i \quad 1\leq i<j\leq n
\end{equation}
in the subspace $U_{2n-2}\subset {\mathbb R}^{2n-1}$ given by
$$U_{2n-2}=\{(x_1,\ldots,x_n,y_1,\ldots,y_{n-1})\in {\mathbb R}^{2n-1}\:
: \: x_1+\cdots +x_n=0\}.$$  
\end{definition}
\begin{remark}
\label{rem:adjust}  
{\em Restricting our arrangement in ${\mathbb R}^{2n-1}$ to the set
$U_{2n-2}$ does not change the 
number of regions, because of the following observation: given a point
$$(x_1,\ldots,x_n,y_1,\ldots,y_{n-1})\in {\mathbb R}^{2n-1},$$
all points of the line
$$
\{(x_1+t,\ldots,x_n+t,y_1,\ldots,y_{n-1})\::\: t\in {\mathbb R}\}
$$
belong to the same region of ${\mathcal H}_{2n-2}$,
  considered as a hyperplane arrangement in ${\mathbb R}^{2n-2}$, since
  $(x_i+t)-(x_j+t)=x_i-x_j$ holds for all $1\leq i<j\leq n$. There is
  exactly one choice of $t$ on this line for which the sum of the
  $x$-coordinates is zero. Intersecting our picture with $U_{2n-2}$
  allows us to get rid of an inessential dimension. It also makes our
  definition more compatible with the usual definition of the Linial
  arrangement, due to the following observation. Intersecting ${\mathcal
    H}_{2n-2}$ with all hyperplanes of the form $y_j=1$
    yields a hyperplane arrangement, which, after discarding the
    redundant $y$-coordinates, is exactly the Linial arrangement
    ${\mathcal L}_{n-1}$.}   
\end{remark}
Next we associate to each region $R$ of the homogenized Linial arrangement
${\mathcal H}_{2n-2}$ a tournament $T(R)$ on $\{1,\ldots,n\}$ as
follows. For each $i<j$, set $i\rightarrow j$ if the points of the
region satisfy $x_i-y_i>x_j$, and set $j\rightarrow i$ if $x_i-y_i<x_j$
holds for all points in the region. The correspondence $R\mapsto T(R)$
is clearly well-defined and injective.

\begin{theorem}
  \label{thm:regions}
The correspondence $R\mapsto T(R)$ establishes a bijection between all
regions of the homogenized Linial arrangement
${\mathcal H}_{2n-2}$ and all alternation acyclic tournaments
on the set $\{1,\ldots,n\}$ 
\end{theorem}
\begin{proof}
First we show that every tournament associated to a region is
alt-acyclic. Assume, by contradiction, that there is a region
$R$, such that the tournament $T(R)$ is not alt-acyclic. By
Corollary~\ref{cor:4cycle} this implies the existence of an alternating
$4$-cycle $i_1\desc i_2\asc i_3\desc i_4\asc i_1$. By the definition of $T(R)$,
all points of the region $R$ satisfy
$$
x_{i_1}>x_{i_2}-y_{i_2}>x_{i_3}\quad\mbox{because of $i_1\desc i_2\asc i_3$, and
}  
$$
$$
x_{i_3}>x_{i_4}-y_{i_4}>x_{i_1}\quad\mbox{because of $i_3\desc i_4\asc i_1$}.   
$$
We obtain the contradiction $x_{i_1}>x_{i_1}$. 

Next we show that every alt-acyclic tournament $T$ on $\{1,\ldots,n\}$
is of the form $T(R)$ for some region $R$. Consider an alt-acyclic
tournament $T$. By Proposition~\ref{prop:alta-arc}, the tournament $T$ is
induced by a biordered forest with code $(\pi, p)$. Let us set
$$x_i=\frac{n+1}{2}-\pi^{-1}(i)\quad\mbox{for $i=1,2,\ldots,n$}$$
and let us set
$$y_i:=\pi^{-1}(p(i))- \pi^{-1}(i)-1/2 \quad\mbox{for
  $i=1,\ldots,n-1$.}$$
Observe first that we have
$$
\sum_{i=1}^n x_i=\sum_{i=1}^n \left(\frac{n+1}{2}-\pi^{-1}(i)\right)=
\frac{n(n+1)}{2}-\sum_{i=1}^{n} i=0,
$$
hence the point $(x_1,\ldots,x_n,y_1,\ldots,y_{n-1})$ belongs to the
vector space $U_{2n-2}$. Observe next, that for each $i<j$ the
difference $x_i-x_j=\pi^{-1}(j)-\pi^{-1}(i)$ is the difference between
the positions of $j$ and $i$. This integer is strictly more than
$y_i=\pi^{-1}(p(i))-\pi^{-1}(i)-1/2$ exactly when $j=p(i)$ or $j$ is to
the right of $p(i)$ in $\pi$. Therefore $T(R)$ is exactly the
tournament induced by the biordered forest whose code is $(\pi,p)$.   
\end{proof}

Now we are ready to prove one of the main results of our paper.

\begin{theorem}
\label{thm:all-alta}
The number of alternation acyclic tournaments on the set
$\{1,\ldots,n\}$ is the median Genocchi number $H_{2n-1}$.
\end{theorem}
\begin{proof}
By Theorem~\ref{thm:regions} the statement is equivalent to showing that
the number of regions in the homogenized Linial arrangement ${\mathcal
  H}_{2n-2}$ is $H_{2n-1}$. We will find this number using Zaslavsky's
formula (\ref{eq:Zaslavsky}), where we compute the characteristic
polynomial using Athanasiadis' result (\ref{eq:Christos}). 
To simplify our calculations, instead of applying (\ref{eq:Christos}) to
the hyperplane arrangement ${\mathcal H}_{2n-2}$ directly, we will count the
regions of the hyperplane arrangement $\widehat{{\mathcal H}}_{2n}$,
given by the equations (\ref{eq:hL}) in ${\mathbb R}^{2n}$ with
coordinate functions $x_1$,\ldots,$x_n$, $y_1$,\ldots,$y_{n}$. In other
words, rather than removing one inessential dimension by restricting to
the subspace $U_{2n-2}$ (keep in mind Remark~\ref{rem:adjust} pointing
out that this restriction does not change the number of regions), we add
an additional inessential dimension $y_n$ that is not involved in the
equations defining the hyperplanes. The proof of
Theorem~\ref{thm:regions} may be applied to show that the number of regions 
is the same as the number of alt-acyclic tournaments on $\{1,\ldots,n\}$
(with the remark that the value of $y_n$ may be chosen in an arbitrary
fashion). Let us now consider the hyperplane arrangement $\widehat{{\mathcal
    H}}_{2n}$ as the subset of  ${\mathbb F}_q^{2n}$ for some very large prime 
$q$. Let us introduce the shorthand notation $\chi(n,k,q)$ for the
number of elements in the set 
$$
{\mathcal S}_{n,k}(q)=\left\{(x_1,\ldots,x_n, y_1,\ldots, y_{n})\in
{\mathbb F}_q^{2n}-\bigcup 
\widehat{\mathcal H}_{2n}\::\: |\{x_1-y_1,\ldots, x_{n}-y_{n}\}|=k
\right\}.
$$
In other words, ${\mathcal S}_{n,k}(q)$ is the set of those points in
${\mathbb F}_q^{2n}-\bigcup \widehat{\mathcal H}_{2n}$, for which the
set\hfill\break  $\{x_1-y_1,\ldots, x_{n}-y_{n}\}$ has $k$ elements. By
(\ref{eq:Christos}), we must have
\begin{equation}
\label{eq:Hsum}  
\chi(\widehat{{\mathcal H}}_{2n},q)=\sum_{k=1}^n \chi(n,k,q).
\end{equation}
We claim that the numbers
$\chi(n,k,q)$ satisfy the recurrence 
\begin{equation}
\label{eq:chirec}  
\chi(n,k,q)=(q-k)\cdot k \cdot \chi (n-1,k,q)+(q-k+1)^2\cdot 
\chi(n-1,k-1,q) \quad\mbox{for $n\geq 2$.}
\end{equation}
Indeed, let us first select the values of  $x_1,\ldots,x_{n-1}$ and
$y_1,\ldots, y_{n-1}$ of a vector belonging to ${\mathcal S}_{n,k}(q)$.
Since the set $\{x_1-y_1,\ldots, x_{n}-y_{n}\}$ has either the same size
as $\{x_1-y_1,\ldots, x_{n-1}-y_{n-1}\}$ or it has just one more
element, the set $\{x_1-y_1,\ldots, x_{n-1}-y_{n-1}\}$ must have $k$ or
$k-1$ elements. Furthermore the coordinates   $x_1,\ldots,x_{n-1}$ and
$y_1,\ldots, y_{n-1}$ do not satisfy those equations
from (\ref{eq:hL}) which do not involve $x_n$ or $y_n$.
Depending on the choice between $k$ and $k-1$, this selection
may be performed in $\chi (n-1,k,q)$ or $\chi (n-1,k-1,q)$ ways,
respectively. In the case when $\{x_1-y_1,\ldots,x_{n-1}-y_{n-1}\}$ has
$k$ elements, there are $q-k$ ways to select the value of $x_n$ from the
complement of the set $\{x_1-y_1,\ldots,x_{n-1}-y_{n-1}\}$. Once this
selection is made, we may select $y_n$ in $k$ ways, making sure that
$x_n-y_n$ belongs to the set $\{x_1-y_1,\ldots,x_{n-1}-y_{n-1}\}$. 
Similarly, in the case when $\{x_1-y_1,\ldots,x_{n-1}-y_{n-1}\}$ has $k-1$
elements, there are $q-k+1$ ways to select the value of $x_n$, and also
$q-k+1$ ways to select the value of $y_n$ afterward. Both $x_n$ and
$x_n-y_n$ must belong to the complement of
$\{x_1-y_1,\ldots,x_{n-1}-y_{n-1}\}$ in this case. Using the initial
condition
\begin{equation}
\label{eq:chiinit}  
\chi(1,k,q)=\delta_{1,k} q^2
\end{equation}  
(where $\delta_{1,k}$ is the Kronecker delta function), the polynomials
$\chi(n,k,q)$ may be computed. Since, for each $n$, the ambient space is
$2n$ dimensional, the number of regions of $\widehat{\mathcal H}_{2n}$
is equal to
$$
(-1)^{2n} \chi\left(\widehat{\mathcal H}_{2n},-1\right)=
\chi\left(\widehat{\mathcal H}_{2n},-1\right)=
\sum_{k=1}^n \chi(n,k,-1).
$$
Introducing $r(n,k):=\chi(n,k,-1)$, the initial condition
(\ref{eq:chiinit}) yields $r(1,1)=1$ and the recurrence
  (\ref{eq:chirec}) may be rewritten as
\begin{equation}
  \label{eq:rrec}
r(n,k)=-(k+1)\cdot k\cdot r(n-1,k) + k^2\cdot r(n-1,k-1).
\end{equation}  
Introducing
$$
PS_n^{(k)}=\frac{(-1)^{n-k}\cdot r(n,k)}{(k!)^2},
$$
the initial condition $r(1,1)=1$ may be transcribed as $PS_1^{(1)}=1$,
and the recurrence (\ref{eq:rrec}) may be rewritten as
\begin{equation}
\label{eq:LSrec}  
PS_n^{(k)}=k(k+1)\cdot PS_{n-1}^{(k)}+PS_{n-1}^{(k-1)}.
\end{equation}
Equation (\ref{eq:LSrec}) is a recurrence relation satisfied by the
Legendre-Stirling numbers, shown by Andrews, Gawronski and
Littlejohn~\cite[Theorem 5.3]{Andrews-LS}, and the initial conditions
also 
match. We obtain that $r(n,k)=(-1)^{n-k}\cdot (k!)^2\cdot
PS_n^{(k)}$, and that
$$
r\left({\mathcal H}_{2n-2}\right)
=r\left(\widehat{{\mathcal H}}_{2n}\right)
=\sum_{k=1}^n (-1)^{n-k}\cdot (k!)^2\cdot PS_n^{(k)}. 
$$
It was shown in~\cite{Claesson-LS} (see Equation~(\ref{eq:HLS}))
that the above sum equals the median Genocchi number $H_{2n-1}$.
\end{proof}

\section{Direct counting using the largest maximum order}
\label{sec:lmax}

By Proposition~\ref{prop:alta-arc}, given an alternation acyclic
tournament $T$,  
after fixing a linear extension $\pi$ of the partial order $\rawalk$,
there is a unique parent function $p$ such that the biordered forest
encoded by $(\pi,p)$ induces $T$. In this section we fix one such
linear extension for each alternation acyclic tournament and
describe how to recognize the valid pairs $(\pi,p)$. This will allow us
to directly count alternation acyclic tournaments of various kinds.

\begin{definition}
For an alternation acyclic tournament $T$ on $\{1,\ldots,n\}$, we define
the {\em largest maximal order} to be the permutation
$\lambda=\lambda(1)\cdots\lambda(n)$, given recursively as follows:
\begin{enumerate}
\item $\lambda(n)$ is the largest maximal element of $\{1,\ldots,n\}$
  ordered by $\rawalk$.
\item Once $\lambda(j)$ has been determined for all $j>k$, $\lambda(k)$
  is the largest maximal element in the poset obtained by restricting
  the partial order $\rawalk$ to $\{1,\ldots,n\}- \{\lambda(k+1),\ldots,\lambda(n)\}$.    
\end{enumerate}   
We call the unique pair
$(\lambda,p)$ inducing $T$ the {\em largest maximal representation of $T$}.  
\end{definition}
Note that the largest maximal order is necessarily a linear extension of
the partial order $\rawalk$, and that each $\lambda(k)$ is the largest
maximal element in the poset obtained by restricting the partial order
$\rawalk$ to the set $\{\lambda(1),\ldots,\lambda(k)\}$.
For example, the largest maximal order for the tournament induced by the
pair $(\pi,p)$ shown in Figure~\ref{fig:arcs} is $125346$, and the
largest maximal representation is shown in
Figure~\ref{fig:arcs-lmax}. It is easy to verify that this diagram
induces the same tournament, the fact that this is the largest maximal
representation will be easily verifiable using
Proposition~\ref{prop:lmaxchar} below. 
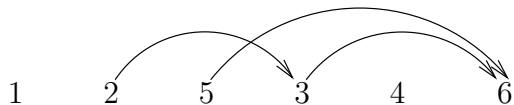
\begin{figure}[h]
\begin{center}
\input{fig02_pspdf.tex}
\end{center}
\caption{Largest maximal representation of the tournament induced in
  Figure~\ref{fig:arcs}}
\label{fig:arcs-lmax}
\end{figure}
\begin{remark}
{\em Consider the largest maximal representation shown in
Figure~\ref{fig:arcs-lmax-rem}. Here $\pi(4)=4$ is the largest maximal
element of the set $\{1,2,5,4\}$ because we have $5\desc 3\asc 4$ and so
$5\srawalk 4$ holds. We only discarded the vertex $3$ from the set of
elements to be considered as a maximal element, but we can not correctly
interpret the restriction of the partial order to the subset
$\{1,2,5,4\}$ without considering the relation of $3$ to $4$ and $5$ in
the entire tournament.}     
\end{remark}
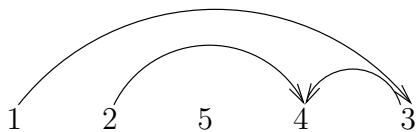
\begin{figure}[h]
\begin{center}
\input{fig03_pspdf.tex}
\end{center}
\caption{Largest maximal representation illustrating the importance of
  ``discarded'' elements}
\label{fig:arcs-lmax-rem}
\end{figure}
The next statement completely characterizes the largest maximal
representations. Recall that $i\in \{1,\ldots, n\}$ is a {\em descent}
of the permutation $\pi$ of $\{1,\ldots, n\}$ if $\pi(i)>\pi(i+1)$
holds. 
\begin{proposition}
  \label{prop:lmaxchar} Given a permutation $\lambda$ of $\{1,\ldots,n\}$
  and a parent function
  $$p:\{1,2,\ldots,n\}\rightarrow \{2,\ldots,n\}\cup\{\infty\},$$
  the pair $(\lambda,p)$ is the largest
  maximal representation of the tournament induced by $(\lambda,p)$ if
  and only if for each descent $i$ of $\lambda$, the vertex
  $\lambda(i+1)$ belongs to the range of $p$. 
\end{proposition}
\begin{proof}
Assume first that $(\lambda,p)$ is a largest maximal representation and
that $i$ is a descent of $\lambda$. By definition $\lambda(i)$ is a
maximal element in the subset $\{\lambda(1),\ldots,\lambda(i)\}$,
ordered by $\rawalk$, but it is not a maximal element in the subset
$\{\lambda(1),\ldots,\lambda(i), \lambda(i+1)\}$, since $\lambda(i+1)$
is the largest maximal element in the latter set, and it is smaller than
$\lambda(i)$. Hence $\lambda(i+1)$ must cover $\lambda(i)$ in the
restriction of $\rawalk$ to $\{\lambda(1),\ldots,
\lambda(i+1)\}$. Since, for any, $k>i$, the relation $\lambda(k)\srawalk
\lambda(i+1)$ can not hold, the relation $\lambda(i)\srawalk \lambda(i+1)$
is also a cover relation in the entire set $\{1,\ldots,n\}$. Thus there
is a $j\in\{1,\ldots,n\}$ such that $\lambda(i)\desc j\asc \lambda(i+1)$
holds. Since $\lambda(i)$ is immediately to the left of $\lambda(i+1)$,
we must have $p(j)=\lambda(i+1)$.

Assume next that an alt-acyclic tournament is induced by a code $(\pi,
p)$, in which for each descent $i$ of $\pi$, the element $\pi(i+1)$
belongs to the range of $p$. We will show by induction on $k$ that for
each $k\in \{1,\ldots,n\}$ the vertex $\pi(k)$ is the largest maximal
element of the set $\{\pi(1),\ldots,\pi(k)\}$. For $k=1$ we must have
$\pi(1)=1$ since setting $1=\pi(i+1)$ for some $i\geq 1$ would make $i$
a descent and $1$ is never in the range of the parent function
$p$. Assume now that the statement holds for some $k$ and consider the
set $\{\pi(1),\ldots,\pi(k+1)\}$. Since, by
Proposition~\ref{prop:arcs-alta}, the 
permutation $\pi$ is a linear extension of the partial order $\rawalk$,
the element $\pi(k+1)$ is a maximal element of the set
$\{\pi(1),\ldots,\pi(k+1)\}$ ordered by $\rawalk$, we only need to show
that it is the largest maximal element. There is nothing to prove when
$\pi(k)<\pi(k+1)$ holds: adding $\pi(k+1)$ to the set  
$\{\pi(1),\ldots,\pi(k)\}$ can only decrease the list of the maximal
elements and, by our induction
hypothesis,  $\pi(k)$ was the largest element on this list before we
added $\pi(k+1)$. We are left to consider the case when $\pi(k)>\pi(k+1)$
holds, that is, $k$ is a descent. By our assumption there is a
$j<\pi(k+1)$ satisfying $\pi(k+1)=p(j)$. Consider any $i\leq k$ for
which $\pi(i)>\pi(k+1)$ holds. This element is to the left of
$\pi(k+1)=p(j)$ and it is larger than $j$. Hence we have $\pi(i)\desc
j\asc \pi(k+1)$, implying $\pi(i)\srawalk \pi(k+1)$. We obtained that no
element of $\{\pi(1),\ldots,\pi(k+1)\}$ that is larger than $\pi(k+1)$
can be a maximal element in this set, with respect to
$\rawalk$. Therefore $\pi(k+1)$ is the largest maximal element.  
\end{proof}  

Proposition~\ref{prop:lmaxchar} allows us to count alt-acyclic tournaments in
a recursive fashion, by using the following {\em reduction} operation.
\begin{definition}
\label{def:reduction}  
Given the largest maximal representation $(\lambda,p)$ of an alternation
acyclic tournament $T$ on $\{1,\ldots,n\}$ for some $n\geq 2$, we define
its {\em reduction} to the set $\{1,\ldots,n-1\}$ to be the alternation
acyclic tournament $T'$ with largest maximal representation
$(\lambda',p')$ where
$$
\lambda'(i)=
\begin{cases}
\lambda(i) & \mbox{if $i< \lambda^{-1}(n)$;}\\
\lambda(i+1) & \mbox{if $i\geq \lambda^{-1}(n)$;}\\
\end{cases}
\quad\mbox{and}\quad
p'(i)=
\begin{cases}
p(i)& \mbox{if $p(i)\neq n$;}\\
\infty & \mbox{if $p(i)= n$.}\\  
\end{cases}  
$$
In other words, the permutation $\lambda'(1)\cdots \lambda'(n-1)$ is
obtained from $\lambda(1)\cdots \lambda(n)$ by deleting the letter $n$,
and the parent function $p'$ is obtained from $p$ by changing all values
$p(i)=n$ to $p'(i)=\infty$.
\end{definition}
\begin{proposition}
\label{prop:reduction}
If $(\lambda,p)$ is the largest maximal representation of an alternation
acyclic tournament $T$ on $\{1,\ldots,n\}$ then the pair $(\lambda',p')$
given in Definition~\ref{def:reduction} is the largest maximal    
representation of an alternation acyclic tournament on
$\{1,\ldots,n-1\}$.
\end{proposition}
\begin{proof}
Clearly $\lambda'$ is a permutation of $\{1,\ldots,n-1\}$, and the
function $p'$ maps $\{1,\ldots,n-1\}$ into $\{2,\ldots,n-1\}\cup
\{\infty\}$ in such a way that $p'(i)>i$ holds for all $i$. We only need
to verify that for every descent $i$ of $\lambda'$, the element
$\lambda'(i+1)$ is in the range of $p'$. This is most easily verified by
visualizing the reduction operation in terms of the arc
representations. In such terms, the reduction
operation removes the letter $n$, and redirects all arrows ending in $n$
to point to $\infty$. If a letter in $\lambda'$ is less than the letter
immediately preceding it, the same remains true even after inserting the
letter $n$. (Note that $\lambda^{-1}(n)$ is a descent unless
$\lambda^{-1}(n)=n$.) Finally the range of $p'$ is obtained from the
range of $p$ by removing $n$ from it (if it was present).     
\end{proof}

\begin{definition}
We say that an alternation acyclic tournament has {\em type $(n,i,j)$}
if it is a tournament on $\{1,\ldots,n\}$, and the parent function $p$ in its
largest maximal representation $(\lambda,p)$ satisfies
$|p^{-1}(\infty)|=i$ and $|p(\{1,\ldots,n\})|=j+1$. We will denote the
number of alternation acyclic tournaments of type $(n,i,j)$ with
$A(n,i,j)$.  
\end{definition}  

Note that $p(n)=\infty$ always holds, so $A(n,i,j)>0$ can only hold when
$i\geq 1$. Similarly, $j\geq 0$ must hold.

\begin{proposition}
\label{prop:altac}
The numbers $A(n,i,j)/j!$ are integers, they are given by the initial
condition 
$A(1,i,j)/j!=\delta_{i,1}\cdot \delta_{0,j}$ (where $\delta_{i,1}\cdot
\delta_{0,j}$ is a product of Kronecker deltas), and the recurrence relation
$$
\frac{A(n,i,j)}{j!}=\sum_{k=i}^{n-1}\binom{k}{i-1}\cdot
\frac{A(n-1,k,j-1)}{(j-1)!}+(j+1)\cdot \frac{A(n-1,i-1,j)}{j!}
\quad\mbox{for $n\geq 2$.} 
$$
\end{proposition}
\begin{proof}
The statement is equivalent to $A(1,i,j)=\delta_{i,1}\cdot \delta_{0,j}$
, and the recurrence relation 
$$
A(n,i,j)=\sum_{k=i}^{n-1}\binom{k}{i-1}\cdot j\cdot
A(n-1,k,j-1)+(j+1)\cdot A(n-1,i-1,j)    \quad\mbox{for $n\geq 2$.}
$$
Suppose we have an alternation acyclic tournament $T$ of type $(n,i,j)$,
and consider its reduction $T'$. We claim that the type of $T'$ must be
either $(n-1,k,j-1)$ for some $k\geq i$ or $(n-1,i-1,j)$. Indeed, if $n$ is in
the range of $p$ then the range of $p'$ has one less element, and 
$p'^{-1}(\infty)=p^{-1}(\infty)-\{n\}\cup p^{-1}(n)$ properly contains
$p^{-1}(\infty)-\{n\}$. If $n$ is not in the range of $p$ then
$p^{-1}(n)=\emptyset$, $p'^{-1}(\infty)=p^{-1}(\infty)-\{n\}$ has
exactly one less element than $p^{-1}(\infty)$, and the range of $p'$
equals the range of $p$.

We claim that any alt-acyclic tournament $T'$ of type $(n-1,k,j-1)$
(where $k\geq i$) is the reduction of exactly $\binom{k}{i-1}\cdot j$
alt-acyclic tournaments of type $(n,i,j)$. Indeed, unless $n$ is
inserted as the last letter of $\lambda$, it creates a descent, so it must be
inserted right before a vertex that is in the range of $p'$. There are
$j$ ways to perform this insertion. Furthermore, we must take a
$(k-i+1)$-element subset of $p'^{-1}(\infty)$ and reassign them to have
$n$ as their parent. A similar, but simpler reasoning shows that for any
alt-acyclic tournament $T'$ of type $(n-1,i-1,j)$ there are exactly
$(j+1)$ alt-acyclic tournaments of type $(n,i,j)$ whose reduction is
$T$. 
\end{proof}

We computed the numbers $A(n,i,j)/j!$ using Maple and the formula given
in Proposition~\ref{prop:altac} for $n\leq 5$. These are given in
Table~\ref{tab:Anij}. 
\begin{table}
\label{tab:altac}  
\begin{tabular}{llll}
$n=2$ &
\begin{tabular}{r|cc}
1 & 1 & \\
0 & 0 & 1\\
\hline
\slashbox{j}{i} & 1 & 2\\  
\end{tabular}
&$n=3$&  
\begin{tabular}{r|ccc}
2 & 1 & &\\
1 & 1 & 4& \\
0 & 0 & 0& 1\\
\hline
\slashbox{j}{i} & 1 & 2& 3\\  
\end{tabular}
\\
\ & \ & \ & \ \\
\\  
$n=4$ &
\begin{tabular}{r|cccc}
3 & 1 &&&\\  
2 & 5 & 11 &&\\
1 & 1 & 5 & 11& \\
0 & 0 & 0& 0 & 1\\
\hline
\slashbox{j}{i} & 1 & 2& 3 & 4\\  
\end{tabular}
& $n=5$&
\begin{tabular}{r|ccccc}
4 & 1 &&&&\\
3 & 16 & 26 &&&\\  
2 & 17 & 58 & 66 && \\
1 & 1 & 6 & 16 & 26 & \\
0 & 0 & 0& 0& 0 & 1\\
\hline
\slashbox{j}{i} & 1 & 2& 3 & 4 & 5\\  
\end{tabular}\\
\end{tabular}
\label{tab:Anij}
\caption{The values of $A(n,i,j)/j!$ for $2\leq n\leq 5$}
\end{table}
A generating function formula for the numbers $A(n,i,j)/j!$ will be
given in Section~\ref{sec:genfun}. Inspecting the tables we can
make several observations, some of which are easy to show.
\begin{proposition}
\label{prop:i+j>n}  
$A(n,i,j)=0$ holds for $i+j>n$.
\end{proposition}  
Indeed, for the largest maximal representation $(\lambda,p)$,  to have
$j+1$ elements in the range of $p$, we need at least $j$ elements of
$\{1,\ldots,n\}$ to have a parent different from $\infty$.
\begin{proposition}
\label{prop:ani0}
  $A(n,i,0)=\delta_{i,n}$ where $\delta_{i,n}$ is the Kronecker delta. 
\end{proposition}  
Indeed, when the range of $p$ is $\{\infty\}$ then all elements have
$\infty$ as their parent. 
It is only a little harder to show that in the main diagonal of each
table we have the {\em Eulerian} numbers.
\begin{proposition}
The number $A(n,n-j,j)/j!$ is the number of permutations of
$\{1,\ldots,n\}$ having exactly $j$ descents.
\end{proposition}  
\begin{proof}
  Because of Proposition~\ref{prop:i+j>n}, when we set $i=n-j$ in the
  recurrence given in Proposition~\ref{prop:altac}, only the term
  associated to $k=n-j$ will have a positive contribution. By
  $\binom{n-j}{n-j-1}=n-j$ we get  
$$
\frac{A(n,n-j,j)}{j!}=(n-j)\cdot
\frac{A(n-1,k,j-1)}{(j-1)!}+(j+1)\cdot \frac{A(n-1,n-j-1,j)}{j!}
$$
for $n\geq2$. This is exactly the recurrence for the Eulerian numbers,
and the initial conditions match. 
\end{proof}

It may be a little harder to notice that the numbers in
the first column multiplied by the factorial of the row index add up to
the Genocchi numbers of the first kind, that is,
\begin{equation}
  \label{eq:gen1}
|G_{2n}|=\sum_{j=0}^{n-1} A(n,1,j).
\end{equation}  
We will prove a generalization of Equation~(\ref{eq:gen1}) in
Section~\ref{sec:refined}. 

\section{Refined counting of alternation acyclic tournaments}
\label{sec:refined}

The main result of this section is the following generalization of
Dumont's theorem, which also refines Theorem~\ref{thm:all-alta}. 
\begin{theorem}
\label{thm:Dumontfine}
For each $i\in \{1,\ldots,n\}$, the sum $\sum_{j=0}^n A(n,i,j)$
is the number of ordered pairs 
$$((a_1,\ldots,a_{n-1}), (b_1,\ldots,b_{n-1}))\in {\mathbb Z}^{n-1}\times
{\mathbb Z}^{n-1}$$ satisfying the following conditions:
\begin{enumerate}
\itemsep=-1pt  
\item $0\leq a_k\leq k$ and $1\leq b_k\leq k$ hold for all $k\in
  \{1,\ldots,n-1\}$;
\item the set $\{a_1,b_1,\ldots,a_{n-1},b_{n-1}\}$  contains
  $\{1,\ldots,n-1\}$;
\item $|\{k\in \{1,\ldots,n-1\} \::\: a_k=0\}|=i-1$.  
\end{enumerate}  
\end{theorem}

\begin{remark}
\label{rem:Dumontfine}  
{\em   
Theorem~\ref{thm:Dumontfine} above is a direct generalization of
Corollary~\ref{cor:dumont}, a {\em restated} variant of Dumont's
original Theorem~\ref{thm:dumont}. This generalization has a shorter
proof. A similar but longer argument would allow generalizing
Theorem~\ref{thm:dumont} to stating that 
$\sum_{j=0}^n A(n,i,j)$ equals the number of excedant functions
$f:\{1,2,\ldots,2n-1\}\rightarrow \{1,2,\ldots,2n-1\}$ satisfying the
following conditions:
\begin{enumerate}
\item $f(2k)\leq 2n-2$ holds for $k=1,\ldots,n-1$; 
\item $f(\{1,2,\ldots,2n-1\})=\{2,4,\ldots,2n-2\}\cup \{2n-1\}$;
\item $|f^{-1}(\{2n-1\})|=i$.
\end{enumerate}
 } 
\end{remark}

The key ingredient to proving Theorem~\ref{thm:Dumontfine} is the
following bijection. 

\begin{theorem}
\label{thm:descbij}
There is a bijection between the set of all permutations $\pi$ of
$\{1,\ldots,n\}$ and the set of excedant functions $f:
\{1,\ldots,n\}\rightarrow \{1,\ldots,n\}$ such that each permutation $\pi$ and
the corresponding excedant function $f$ have the following property:
$$
\{1,\ldots,n\}-\{f(1),\ldots,f(n)\}=\{\pi(i+1) \::\: 1\leq i\leq n-1\quad
\mbox{and}\quad \pi(i)>\pi(i+1)\}. 
$$
\end{theorem}  
\begin{proof}
We will describe our bijection using the process of inserting the
numbers $1,2,\ldots,n$ into the permutation $\pi$ in decreasing order. In
order to reduce the number of cases, we place before the first
number $\pi(1)$ the number $\pi(0):=0$ and after the last number
$\pi(n)$ the number $\pi(n+1)=n+1$. Thus every number is inserted
between two numbers. For example, for $n=6$ and the permutation
$\pi(1)\cdots\pi(6)=615342$ we have the insertion process
$$
07\rightarrow 067 \rightarrow 0657 \rightarrow 
06547 \rightarrow 065347 \rightarrow 0653427 \rightarrow 06153427.
$$
The number $f(i)$ is computed in step $n+1-i$ when we insert $i$ into the
permutation between the numbers $u$ and $v$, using the following rule:
\begin{equation}
\label{eq:fdef}  
f(i):=
\begin{cases}
v & \mbox{if $u>v$;}\\
\overleftarrow{u} & \mbox{if $0<u<v$};\\
i & \mbox{if $u=0$.}\\
\end{cases}  
\end{equation}
Here $\overleftarrow{u}$ is the leftmost number $w$ in the current word
such that the consecutive subword $w\cdots u$ is {\em decreasing}, that is, each
number in it is smaller than the immediately preceding number. (We
have $\overleftarrow{u}=u$ exactly when $u$ is immediately preceded by a smaller
number.) In our example we have $(f(1),\ldots, f(6))=(5,4,4,6,6,6)$. In
the third step, when we inserted $4$ between $5$ and $7$, we set
$f(4)=\overleftarrow{5}=6$, in the fifth step, when we inserted $2$
between $4$ and $7$, we set $f(2)=\overleftarrow{4}=4$. In the last
step, when we inserted $1$ between $6$ and $5$, we set $f(1)=5$.

The operation $\pi\mapsto f$ is
well-defined. The numbers $f(i)$ clearly satisfy $i\leq f(i)\leq
n$. Since the number of all words $\pi(1)\cdots\pi(n)$ is the same as
the number of all excedant functions $f:\{1,\ldots,n\} \rightarrow
\{1,\ldots,n\}$, to show that we defined a bijection, it suffices to
show that our  assignment is injective: there is at most one way to
reconstruct a permutation from an excedant function $f$.

We always have $f(n)=n$ and the first step is to insert $n$ between $0$
and $n+1$, the last line of the definition (\ref{eq:fdef}) is
applicable. Assume, by induction, that there is only one way to
reconstruct the insertion of $n$, $n-1$, \ldots, $i+1$, based on the
knowledge of $f(n)$, $f(n-1)$, \ldots, $f(i+1)$.  Consider $f(i)$
satisfying $i\leq f(i)\leq n$, and let us show that there is only
one way to insert $i$ that yields the given value of $f(i)$. Only
the last line of the definition (\ref{eq:fdef}) allows setting
$f(i)=i$, the value of $f(i)$ is greater than $i$ on the other
two lines. Thus, in the case when $f(i)=i$, we must insert $i$
right after $0$ as the first new number in our permutation. From now on
we may assume that $f(i)=w$ for some $w>i$. Let $w'$ be the
immediate predecessor of $w$ in our current word. We distinguish two
cases depending on how $w'$ and $w$ compare. If $w'>w$ then $i$
can not be inserted anywhere after $w$, since the only way to obtain
$f(i)=w$ would be to insert $i$ between some $u$ and $v$
satisfying $u<v$ and $w=\overleftarrow{u}$. This is impossible: if
$w\cdots u$ is a decreasing subword, then so is $w'w\cdots u$ and so
$\overleftarrow{u}$ is either $w'$ or an even earlier number. Thus $i$
must be inserted somewhere before $w$, and the only way to get
$f(i)=w$ when $w$ is a number to the right of the place of insertion
is to insert $i$ right before $w$. We are left with the case when
$w'<w$. If we insert $i$ anywhere before $w$, we can not get
$f(i)=w$, only $w'$ or a number to the left of it. We must therefore
insert $i$ after $w$ in such a way that the second line of
(\ref{eq:fdef}) can be used and it yields $f(i)=w$. We must find a $u$
such that the $v$ succeeding $u$ is larger than $u$ and the subword
$w\cdots u$ is decreasing. In other words, we must take the rightmost $u$
such that $w\cdots u$ is a decreasing consecutive subword.

We are left to show that the set $\{f(1),\ldots,f(n)\}$ contains all
numbers between $1$ and $n$ except those values that are immediately
preceded by a larger number in the permutation $\pi(1)\cdots \pi(n)$. We
prove the following generalization of this statement by induction: at
step $n+1-i$ of the insertion process, the set $\{f(i),\ldots, f(n)\}$
contains all elements of the set $\{i,\ldots,n\}$ except those
numbers, which are immediately preceded by a larger number in the
current permutation of $n, n-1, \ldots, i$. At
the first step $n$ is inserted and it is preceded by a smaller
number. We set $f(n)=n$. Assume the statement is true up to step $n-i$
and consider the insertion of $i$. If $i$ is inserted right after
$0$, the current set of numbers immediately preceded by a smaller number
does not change, and $f(i)=i$ is added to the set
$\{f(i+1), \ldots, f(n)\}$. In all other cases $i$ is inserted right after
a larger number and $i\not \in \{f(i),\ldots,f(n)\}$. If $i$ is
inserted between $u$ and $v$ satisfying $u>v$, then $v$ which was
hitherto immediately preceded by a larger number, it is now immediately
preceded by the smaller number $i$. The set of numbers immediately
preceded by a larger number gains $i$ as a new element and loses $v$
as an element, no other change occurs. This change is properly reflected
in setting $f(i)=v$. Finally, if $u<v$ holds, then the only change to
the set of numbers immediately preceded by a larger number is the
addition of $i$ to this set. This is properly handled, if we select
$f(i)$ to be a number that is already present in the set
$\{f(i+1),\ldots,f(n)\}$. Selecting $f(i)=\overleftarrow{u}$ fits the
bill, as $\overleftarrow{u}$ can not be immediately preceded by a larger
number.
\end{proof}

\begin{definition}
We call the excedant function $f:\{1,\ldots,n\}\rightarrow
\{1,\ldots,n\}$ associated to the permutation $\pi$ by the algorithm
described in the proof of Theorem~\ref{thm:descbij} the {\em
  descent-sensitive code} of the permutation $\pi$. 
\end{definition}  

{\em Proof of Theorem~\ref{thm:Dumontfine}:} 
Consider the largest maximal representation $(\lambda,p)$ of an
alternation acyclic 
tournament and let us replace $\lambda$ with its descent-sensitive code
$f:\{1,\ldots,n\}\rightarrow \{1,\ldots,n\}$. Note that we must have 
$\lambda(1)=1$ for the largest maximal order of each alt-acyclic
tournament and this equality is equivalent to $f(1)=1$. Hence the
function $f$ is completely defined by its restriction $\widetilde{f}$ to the set
$\{2,\ldots,n\}$, which sends this set into itself. Similarly, we must
have $p(n)=\infty$, thus the restriction of $p$ to $\{1,\ldots,n-1\}$,
which is a function $\widetilde{p}: \{1,\ldots,n-1\}\rightarrow
\{2,\ldots,n\}\cup\{\infty\}$, completely determines $p$.
Let us define the vectors $(a_1,\ldots,a_{n-1})$ and $(b_1,\ldots,b_{n-1})$
by setting  
$$
a_{k}=
\begin{cases}
  n+1-p(n-k)& \mbox{if $p(n-k)\neq \infty$,}\\
  0 & \mbox{if $p(n-k)= \infty$}\\
\end{cases}  
$$
and $b_k=n+1-f(n+1-k)$ for $k=1,\ldots,n-1$.
The condition $p(n-k)\in \{n-k+1,\ldots,n\}\cup\{\infty\}$ is equivalent to
$0\leq a_k\leq k$ and the condition $n+1-k\leq f(n+1-k)\leq n$ is
equivalent to $1\leq b_k\leq k$. The description given in
Proposition~\ref{prop:lmaxchar} may be restated as follows: the pair of
functions $(f,p)$ comes from a largest maximal code $(\lambda, p)$ if
and only if we have
\begin{equation}
\label{eq:pf}  
\{p(1),\ldots,p(n-1)\}\cup \{f(2),\ldots, f(n)\}\supseteq \{2,\ldots,n\}. 
\end{equation}
This is equivalent to Condition (2) in our statement. Finally, $|\{k\in
\{1,\ldots,n-1\} \::\: a_k=0\}|$ is clearly the number of elements sent
into $\infty$ by $p$.   
\qed

An important special instance of Theorem~\ref{thm:Dumontfine} is the
case $i=1$. In this case all $a_k>0$ is satisfied for all $k$ and the
pairs $((a_1,\ldots,a_{n-1}), (b_1,\ldots,b_{n-1}))$ are exactly the
ones that are counted in
Corollary~\ref{cor:dumont}. Equation~(\ref{eq:gen1}) follows.  

\section{New combinatorial models for the Genocchi numbers}
\label{sec:genmod}
 
Equation~(\ref{eq:gen1}) inspires introducing {\em ascending alternation acyclic
  tournaments}.  

\begin{definition}
We call an alternation acyclic tournament $T$ on $\{1,\ldots,n\}$
{\em ascending} if every $i<n$ is the tail of an ascent, that is, for
each $i<n$ there is a $j>i$ such that $i\rightarrow j$. 
\end{definition}  

\begin{lemma}
\label{lem:asc}  
An alternating acyclic tournament $T$ on $\{1,\ldots,n\}$ is ascending
if and only if it has type $(n,1,j)$ for some $j$. 
\end{lemma}  
Indeed, for any biordered forest  inducing $T$, if $(\pi,p)$ is the code
of the biordered forest, $p(i)=\infty$ holds if and only if $i$ is not the
tail of any ascent. An alt-acyclic tournament is ascending if and only
if $n$ is the only element of $\{1,\ldots,n\}$ whose parent is $\infty$. 

Because of Lemma~\ref{lem:asc}, Equation~(\ref{eq:gen1}) may be
rephrased as follows.

\begin{corollary}
  \label{cor:asc}
The number of ascending alternation acyclic tournaments on
$\{1,\ldots,n\}$ is the unsigned Genocchi number of the first kind $|G_{2n}|$.  
\end{corollary}  

Taking into account Theorem~\ref{thm:all-alta},
Theorem~\ref{thm:Dumontfine} implies the following result on the median
Genocchi numbers. 
\begin{corollary}
  \label{cor:Dumont2}
The median Genocchi number $H_{2n-1}$ is the total number of all ordered
pairs 
$$((a_1,\ldots,a_{n-1}), (b_1,\ldots,b_{n-1}))\in {\mathbb Z}^{n-1}\times
{\mathbb Z}^{n-1}$$
such that $0\leq a_k\leq k$ and $1\leq b_k\leq k$ hold for all $k$
and the set $\{a_1,b_1,\ldots,a_{n-1},b_{n-1}\}$ contains $\{1,\ldots,n-1\}$.
\end{corollary}  
Corollary~\ref{cor:Dumont2} makes the divisibility of $H_{2n-1}$ by
$2^{n-1}$ especially transparent. Furthermore, it inspires the
following model for the normalized median Genocchi numbers. 
\begin{theorem}
\label{thm:nmGenocchi}  
The normalized median Genocchi number $h_n$ is the number of sequences
$\{u_1,v_1\}, \{u_2,v_2\}, \ldots, \{u_n,v_n\}$ subject to the following
  conditions:
\begin{enumerate}
\item the set $\{u_k,v_k\}$ is a (one- or two-element) subset of
  $\{1,\ldots,k\}$;
\item the set $\{u_1,v_1,u_2,v_2,\ldots,u_n,v_n\}$ equals $\{1,\ldots,n\}$.  
\end{enumerate}    
\end{theorem}
\begin{proof}  
By Corollary~\ref{cor:Dumont2}, the median Genocchi number is the number
of pairs of vectors $((a_1,\ldots,a_n), (b_1,\ldots,b_n))$ such that
$0\leq a_k\leq k$ and $1\leq b_k\leq k$ hold for all $k$ and the set
$\{a_1,b_1,\ldots,a_{n},b_{n}\}$ contains $\{1,\ldots,n\}$. Let us
first define a ${\mathbb Z}_2^n$-action of the set of all such
vectors. We define the involution $\phi_k$ for $k\in \{1,\ldots,n\}$ as
follows. The map $\phi_k$ sends
$$
((a_1,\ldots,a_n),
(b_1,\ldots,b_n))\quad\mbox{into}
$$
$$
((a_1,\ldots,a_{k-1},a_k',a_{k+1},\ldots,a_n),(b_1,\ldots,b_{k-1},b_k',b_{k+1},\ldots,b_n)) 
$$
where
$$
(a_k',b_k')=  
\begin{cases}
  (b_k,a_k) & \mbox{if $a_k\neq b_k$ and $a_k\neq 0$;} \\
  (0,b_k) & \mbox{if $a_k=b_k$;} \\
  (b_k,b_k) & \mbox{if $a_k=0$.} \\  
\end{cases}  
$$
In other words, the map $\phi_k$ changes only the $k$-th coordinates of
$(a_1,\ldots,a_n)$ and $(b_1,\ldots,b_n)$, it swaps $a_k$ and $b_k$ if
$\{a_k,b_k\}$ is a two element subset of $\{1,\ldots,k\}$ and it swaps
the pair $(b_k,b_k)$ with the pair $(0,b_k)$. Note that in this second case, we
have 
$$
\{a_k,b_k\}\cap \{1,\ldots,k\}=\{b_k\}
$$
for $a_k=0$, as well as for $a_k=b_k$. The action of the involutions
$\phi_k$ is free, as they act on different coordinates. An orbit
representative for this action is the sequence of sets
$$\{a_1,b_1\}\cap\{1\}, \{a_2,b_2\}\cap \{1,2\},\ldots, \{a_n,b_n\}\cap
\{1,\ldots,n\}.$$
In the case when $a_k\neq 0$ we may
set $u_k=a_k$ and $v_k=b_k$, and in the case when $a_k=0$, we may set
$u_k=b_k$ and $v_k=b_k$.
This orbit representative is valid if and only if the set
$\{u_1,v_1,u_2,v_2,\ldots,u_n,v_n\}$ equals $\{1,\ldots,n\}$.   
\end{proof}
\begin{remark}
{\em It was recently shown by A.\ Bigeni~\cite{Bigeni-bij} that the
  above model is bijectively equivalent to the model introduced by
  Feigin~\cite{Feigin-deg}. The bijection is highly nontrivial,
  A.\ Bigeni's entire paper is devoted to it. It is through Feigin's
  model that the above model is related to the earlier models by 
Kreweras~\cite{Kreweras} and by Kreweras and
Barraud~\cite{Kreweras-Barraud}. All earlier models are related to the
Kreweras triangle~\cite{Kreweras}, and there is an interpretation of the
numbers in the Kreweras triangle through Bigeni's bijection. 
Note, however, that the numbers $A(n,i,j)$
introduced in this paper form a three dimensional array, and they are
{\em not directly related} to the Kreweras triangle.}

\end{remark}  

\begin{remark}
  \label{rem:Dumont}
{\em The variant of Theorem~\ref{thm:Dumontfine}, together with
Theorem~\ref{thm:all-alta} imply the following variant of
Corollary~\ref{cor:Dumont2}: the median Genocchi number $H_{2n-1}$ is
the number of excedant functions
$f:\{1,2,\ldots,2n-1\}\rightarrow \{1,2,\ldots,2n-1\}$ satisfying
$f(2k)\leq 2n-2$ for $k=1,\ldots,n-1$ and 
$$
f(\{1,2,\ldots,2n-1\})=\{2,4,\ldots,2n-2\}\cup
\{2n-1\}.
$$}
\end{remark}

\section{Generating function formulas}
\label{sec:genfun}

In this section we prove a generating function formula for the numbers
$A(n,i,j)$ introduced in Section~\ref{sec:lmax} and obtain the ordinary
generating functions of the Genocchi numbers of both kinds. 

We begin with introducing the generating function
$$\alpha(x,y,t)=\sum_{n=0}^\infty \sum_{i=1}^n \sum_{j=0}^{n-1}
\frac{A(n,i,j)}{j!} x^i y^j t^n,$$
in which we denote the coefficient of $t^n$ by $\alpha_n(x,y)$. 
Proposition~\ref{prop:altac} may be rewritten as
\begin{equation}
  \label{eq:ainit}
  \alpha_1(x,y)=x \quad\mbox{and}
\end{equation}
\begin{equation}
  \label{eq:arec}
  \alpha_{n+1}(x,y)
  =xy (\alpha_n(x+1,y)-\alpha_n(x,y))+x\alpha_n(x,y)
  +xy \frac{\partial}{\partial y}\alpha_n(x,y)
  \quad\mbox{for $n\geq 1$.}
\end{equation}
These equations gain a simpler form after introducing
the formal power series
$$
\beta_n(x,y)=\alpha_n(x,y)\cdot e^{-y}.
$$
For these, equations~(\ref{eq:ainit}) and (\ref{eq:arec}) may be
rewritten as
\begin{equation}
  \label{eq:binit}
  \beta_1(x,y)=xe^{-y} \quad\mbox{and}
\end{equation}
\begin{equation}
  \label{eq:brec}
  \beta_{n+1}(x,y)
  =xy \beta_n(x+1,y)+x\beta_n(x,y)
  +xy \frac{\partial}{\partial y}\beta_n(x,y)
  \quad\mbox{for $n\geq 1$.}
\end{equation}
Let us define the polynomial $\beta_{n,k}(x)$ as the coefficient of
$y^k$ in $\beta_n(x,y)$. Equations~(\ref{eq:binit}) and (\ref{eq:brec}) may be
transformed into  
\begin{equation}
  \label{eq:bkinit}
  \beta_{1,k}(x)=x\cdot\frac{(-1)^k}{k!} \quad\mbox{for $k\geq 0$ and}
\end{equation}
\begin{equation}
  \label{eq:bkrec}
  \beta_{n+1,k}(x)
  =x(\beta_{n,k-1}(x+1)+(k+1)\beta_{n,k}(x))
  \quad\mbox{for $n\geq 1$ and $k\geq 0$.}
\end{equation}
Note that (\ref{eq:bkrec}) also holds for $k=0$, once we set
$\beta_{n,-1}(x)=0$ for all $n$. Let us set finally
$\gamma_{n,k}(x)=\beta_{n,k}(x-k)$.  
Equations~(\ref{eq:binit}) and (\ref{eq:brec}) may be
transformed into  
\begin{equation}
  \label{eq:ckinit}
  \gamma_{1,k}(x,y)=(x-k)\cdot\frac{(-1)^k}{k!} \quad\mbox{for $k\geq 0$ and}
\end{equation}
\begin{equation}
  \label{eq:ckrec}
  \gamma_{n+1,k}(x)
  =(x-k)(\gamma_{n,k-1}(x)+(k+1)\gamma_{n,k}(x))
  \quad\mbox{for $n\geq 1$ and $k\geq 0$.}
\end{equation}
Again we set $\gamma_{n,-1}(x)=0$ for all $n$.
This is an array of polynomials that is easy to compute after
introducing
$$
\gamma_k(x,t)=\sum_{n=1}^{\infty} \gamma_{n,k}(x) t^n.   
$$
For $k=0$, Equation~(\ref{eq:ckinit}) and repeated use of
Equation~(\ref{eq:ckrec}) yields $\gamma_{n,0}=x^n$ for $n\geq 1$. Hence
we have 
\begin{equation}
  \label{eq:cinit}
\gamma_0(x,t)=\frac{xt}{1-xt}.
\end{equation}
For $k\geq 1$, Equation~(\ref{eq:ckrec}) implies the recurrence
\begin{equation}
  \label{eq:crec}
\gamma_{k}(x,t)=\frac{(x-k)t}{1-(x-k)(k+1)t}\cdot \left(\frac{(-1)^k}{k!}+
\gamma_{k-1}(x,t)\right). 
\end{equation}  
Using Equations~(\ref{eq:cinit}) and (\ref{eq:crec}), an easy induction
on $k$ implies
\begin{equation}
  \label{eq:cfun}
\gamma_k(x,t)=\sum_{i=0}^k \frac{(-1)^{k-i}}{(k-i)!}
\prod_{\ell=0}^i \frac{(x-k+\ell) t}{
  1-(x-k+\ell)(k+1-\ell)t}.  
\end{equation}  
Next we introduce
$$
\widetilde{\beta}_k(x,t)=\sum_{n=0}^{\infty} \beta_{n,k}(x) t^n.  
$$
The definition of $\gamma_{n,k}(x)$ implies
$\beta_{n,k}(x)=\gamma_{n,k}(x+k)$ and
$\widetilde{\beta}_k(x,t)=\gamma_k(x+k,t)$. Hence
Equation~(\ref{eq:cfun}) may be rewritten as   
\begin{equation}
  \label{eq:bfun}
\widetilde{\beta}_k(x,t)=\sum_{i=0}^k \frac{(-1)^{k-i}}{(k-i)!}
\prod_{\ell=0}^i \frac{(x+\ell) t}{
  1-(x+\ell)(k+1-\ell)t}.  
\end{equation}  
Finally, as an immediate consequence of the definitions we have
$$
\alpha(x,y,t)=\sum_{k=0}^{\infty} \widetilde{\beta}_k(x,t)\cdot y^k\cdot e^y.   
$$
Combining the last equation with Equation~(\ref{eq:bfun}) we obtain the
formula
$$
\alpha(x,y,t)=\sum_{j=0}^\infty y^j \sum_{k=0}^j \frac{1}{(j-k)!} \sum_{i=0}^k \frac{(-1)^{k-i}}{(k-i)!}
\prod_{\ell=0}^i \frac{(x+\ell) t}{
  1-(x+\ell)(k+1-\ell)t}   
$$
Taking into account $\prod_{\ell=0}^i (k+1-\ell)=(k+1)!/(k-i)!$
we obtain the following result.
\begin{theorem}
 \label{thm:agenfun}  
The generating function $\alpha(x,y,t)=\sum_{n,i,j} A(n,i,j) x^i y^j
t^n/j! $ is given by
$$
\alpha(x,y,t)=\sum_{j=0}^\infty \frac{y^j}{j!}
\sum_{k=0}^j \binom{j}{k} \frac{1}{k+1} \sum_{i=0}^k (-1)^{k-i}
\prod_{\ell=0}^i \frac{(x+\ell)(k+1-\ell)t}{
  1-(x+\ell)(k+1-\ell)t}.
$$
\end{theorem}  
By Theorem~\ref{thm:all-alta}, the generating function of the median Genocchi
numbers $H_{2n-1}$ is obtained by substituting $x=1$ and replacing each
$y^j$ with $j!$ in $\alpha(x,y,t)$. 
\begin{corollary}
 The median Genocchi numbers satisfy 
  $$
 \sum_{n=1}^\infty H_{2n-1} t^n=
 \sum_{j=0}^\infty
 \sum_{k=0}^j \binom{j}{k} \frac{1}{k+1} \sum_{i=0}^k (-1)^{k-i}
\prod_{\ell=0}^i \frac{(1+\ell)(k+1-\ell)t}{
  1-(1+\ell)(k+1-\ell)t}.
  $$
\end{corollary}  
By Corollary~\ref{cor:asc}, the generating function of the Genocchi
numbers of the first kind is obtained by replacing each $y^j$ by $j!$
and then taking the coefficient of $x$ in in
$\alpha(x,y,t)$. To use Theorem~\ref{thm:agenfun}, observe that
all powers of $x$ occur in the products of the form
$$
\prod_{\ell=0}^i \frac{(x+\ell)(k+1-\ell)t}{
  1-(x+\ell)(k+1-\ell)t}.
$$
Here, for $\ell=0$ , the factor
$$
\frac{x(k+1)t}{1-x(k+1)t}=x(k+1)t+x^2(k+1)^2t^2\cdots 
$$
has no constant term, and the coefficient of $x$ is $(k+1)t$. We can
take out this factor, simplify by (k+1), and only the constant terms of
the remaining factors contribute to the coefficient of
$x$. Theorem~\ref{thm:agenfun} thus has the following consequence.
\begin{corollary}
The Genocchi numbers of the first kind
satisfy
$$
\sum_{n=1}^{\infty} |G_{2n}| t^n=
t\cdot \sum_{j=0}^\infty 
\sum_{k=0}^j \binom{j}{k} \sum_{i=0}^k (-1)^{k-i}
\prod_{\ell=1}^i \frac{\ell(k+1-\ell)t}{
  1-\ell(k+1-\ell)t}.
$$
In this formula, when $i=0$, we define the empty product to be equal to $1$. 
\end{corollary}  

\section{Concluding remarks}
\label{sec:concl}

Dumont's first permutation models for the Genocchi numbers were created
by finding a class of excedant functions first~\cite[Corollaire du
  Th\'eor\`eme 3]{Dumont}, and then establishing a bijection between
excedant functions and permutations~\cite[Section 5]{Dumont}. This
bijection is very different from, although similar in spirit to our 
Theorem~\ref{thm:descbij}. Using the bijection presented in
Theorem~\ref{thm:descbij}, new classes of permutations counted by
Genocchi numbers of the first kind may be introduced, however these
classes will be very similar if not identical to the examples obtained
by Dumont, after combining his bijection with Foata's fundamental
transformation~\cite{Foata} which transforms counting excedances into counting
descents. Dumont's bijection between permutations and excedant
functions makes identifying excedances easy, whereas our bijection is
poised on identifying descents. More interesting results could be hoped
for by finding new permutation models for median Genocchi numbers using
Remark~\ref{rem:Dumont} and Corollary~\ref{cor:Dumont2}. The curiosity of all
results presented in this paper is that objects counted by Genocchi
numbers of the first kind are presented as subsets of objects counted by
median Genocchi numbers: it is usually the other way around in the
literature. 

This paper arose in a search for generalizations of semiacyclic
tournaments that appear in the work of Postnikov and
Stanley~\cite{Postnikov-Stanley}. In particular, we have found a
hyperplane arrangement, whose regions are counted by the median Genocchi
numbers, known to be multiples of powers of $2$. Semiacyclic tournaments
count regions in the Linial arrangement, which is a section of the
arrangement we presented in this paper. The number of semiacyclic
tournaments on $n$ vertices is known to be
 $$
2^{-n}\sum_{k=0}^n \binom{n}{k} (k+1)^{n-1}
$$
see~\cite[Theorem 8.1]{Postnikov-Stanley}. It is hard to miss in the
above formula that the sum after 
the factor $2^{-n}$ is obviously an integer, but not obviously a
multiple of $2^n$. No combinatorial proof of this divisibility is
known, perhaps the $q$-counting of the regions of the Linial arrangement
by Athanasiadis~\cite{Athanasiadis-extLin} comes closest. Perhaps the
$q$-counting of the regions 
of our homogenized Linial arrangement, combined with a better
understanding how the Linial arrangement appears as a section of our
arrangement could help find some additional explanations how divisibility by
a power of $2$ appears in both settings.

\section*{Acknowledgments}
This work was partially supported by grants from the Simons Foundation
(\#245153 and \#514648 to G\'abor Hetyei). The author thanks Ange
Bigeni and two anonymous referees for valuable advice, many great
suggestions and important corrections.

\end{document}

%% file: fig01_pspdf.tex
\begin{picture}(0,0)%
\includegraphics{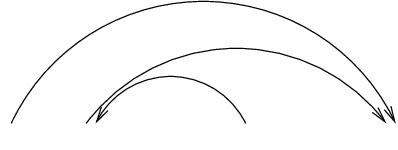}%
\end{picture}%
\setlength{\unitlength}{3947sp}%
\begingroup\makeatletter\ifx\SetFigFont\undefined%
\gdef\SetFigFont#1#2#3#4#5{%
  \reset@font\fontsize{#1}{#2pt}%
  \fontfamily{#3}\fontseries{#4}\fontshape{#5}%
  \selectfont}%
\fi\endgroup%
\begin{picture}(3173,1199)(4411,-3875)
\put(5101,-3811){\makebox(0,0)[lb]{\smash{{\SetFigFont{12}{14.4}{\familydefault}{\mddefault}{\updefault}{\color[rgb]{0,0,0}$3$}%
}}}}
\put(5701,-3811){\makebox(0,0)[lb]{\smash{{\SetFigFont{12}{14.4}{\familydefault}{\mddefault}{\updefault}{\color[rgb]{0,0,0}$1$}%
}}}}
\put(6301,-3811){\makebox(0,0)[lb]{\smash{{\SetFigFont{12}{14.4}{\familydefault}{\mddefault}{\updefault}{\color[rgb]{0,0,0}$2$}%
}}}}
\put(6901,-3811){\makebox(0,0)[lb]{\smash{{\SetFigFont{12}{14.4}{\familydefault}{\mddefault}{\updefault}{\color[rgb]{0,0,0}$4$}%
}}}}
\put(7501,-3811){\makebox(0,0)[lb]{\smash{{\SetFigFont{12}{14.4}{\familydefault}{\mddefault}{\updefault}{\color[rgb]{0,0,0}$6$}%
}}}}
\put(4426,-3811){\makebox(0,0)[lb]{\smash{{\SetFigFont{12}{14.4}{\familydefault}{\mddefault}{\updefault}{\color[rgb]{0,0,0}$5$}%
}}}}
\end{picture}%

%% file: fig02_pspdf.tex
\begin{picture}(0,0)%
\includegraphics{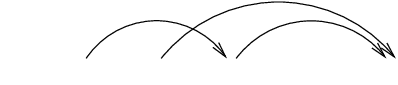}%
\end{picture}%
\setlength{\unitlength}{3947sp}%
\begingroup\makeatletter\ifx\SetFigFont\undefined%
\gdef\SetFigFont#1#2#3#4#5{%
  \reset@font\fontsize{#1}{#2pt}%
  \fontfamily{#3}\fontseries{#4}\fontshape{#5}%
  \selectfont}%
\fi\endgroup%
\begin{picture}(3173,674)(4561,-3875)
\put(6376,-3811){\makebox(0,0)[lb]{\smash{{\SetFigFont{12}{14.4}{\familydefault}{\mddefault}{\updefault}{\color[rgb]{0,0,0}$3$}%
}}}}
\put(5176,-3811){\makebox(0,0)[lb]{\smash{{\SetFigFont{12}{14.4}{\familydefault}{\mddefault}{\updefault}{\color[rgb]{0,0,0}$2$}%
}}}}
\put(5776,-3811){\makebox(0,0)[lb]{\smash{{\SetFigFont{12}{14.4}{\familydefault}{\mddefault}{\updefault}{\color[rgb]{0,0,0}$5$}%
}}}}
\put(6976,-3811){\makebox(0,0)[lb]{\smash{{\SetFigFont{12}{14.4}{\familydefault}{\mddefault}{\updefault}{\color[rgb]{0,0,0}$4$}%
}}}}
\put(4576,-3811){\makebox(0,0)[lb]{\smash{{\SetFigFont{12}{14.4}{\familydefault}{\mddefault}{\updefault}{\color[rgb]{0,0,0}$1$}%
}}}}
\put(7651,-3811){\makebox(0,0)[lb]{\smash{{\SetFigFont{12}{14.4}{\familydefault}{\mddefault}{\updefault}{\color[rgb]{0,0,0}$6$}%
}}}}
\end{picture}%

%% file: fig03_pspdf.tex
\begin{picture}(0,0)%
\includegraphics{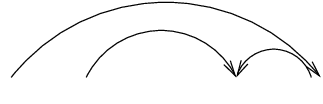}%
\end{picture}%
\setlength{\unitlength}{3947sp}%
\begingroup\makeatletter\ifx\SetFigFont\undefined%
\gdef\SetFigFont#1#2#3#4#5{%
  \reset@font\fontsize{#1}{#2pt}%
  \fontfamily{#3}\fontseries{#4}\fontshape{#5}%
  \selectfont}%
\fi\endgroup%
\begin{picture}(2573,824)(4636,-3875)
\put(7126,-3811){\makebox(0,0)[lb]{\smash{{\SetFigFont{12}{14.4}{\familydefault}{\mddefault}{\updefault}{\color[rgb]{0,0,0}$3$}%
}}}}
\put(6451,-3811){\makebox(0,0)[lb]{\smash{{\SetFigFont{12}{14.4}{\familydefault}{\mddefault}{\updefault}{\color[rgb]{0,0,0}$4$}%
}}}}
\put(5851,-3811){\makebox(0,0)[lb]{\smash{{\SetFigFont{12}{14.4}{\familydefault}{\mddefault}{\updefault}{\color[rgb]{0,0,0}$5$}%
}}}}
\put(5251,-3811){\makebox(0,0)[lb]{\smash{{\SetFigFont{12}{14.4}{\familydefault}{\mddefault}{\updefault}{\color[rgb]{0,0,0}$2$}%
}}}}
\put(4651,-3811){\makebox(0,0)[lb]{\smash{{\SetFigFont{12}{14.4}{\familydefault}{\mddefault}{\updefault}{\color[rgb]{0,0,0}$1$}%
}}}}
\end{picture}%